\documentclass[11pt,reqno,tbtags]{amsart}
\usepackage{graphicx} 
\usepackage{amsmath,amsfonts,amssymb}
\usepackage[margin=1in]{geometry}
\usepackage{url}
\usepackage[square,numbers]{natbib}
\bibpunct[, ]{[}{]}{;}{n}{,}{,}

\numberwithin{equation}{section}

\newcommand{\ZB}{\mathbb{Z}}
\newcommand{\RB}{\mathbb{R}}

\newcommand{\NB}{\mathbb{N}}
\newcommand{\EB}{\mathbb{E}}
\newcommand{\PB}{\mathbb{P}}
\newcommand{\GC}{\mathcal{G}}
\newcommand{\FC}{\mathcal{F}}

\newcommand{\AC}{\mathcal{A}}

\newcommand{\BC}{\mathcal{B}}
\newcommand{\YC}{\mathcal{Y}}
\newcommand{\XC}{\mathcal{X}}
\newcommand{\PC}{\mathcal{P}}

\newcommand{\XF}{\mathfrak{X}}
\newcommand{\YF}{\mathfrak{Y}}

\newcommand{\ignore}[1]{}

\usepackage{amsmath}
\usepackage{amssymb}
\usepackage{mathtools}
\usepackage{amsthm}
\usepackage{tabularx}
\usepackage{enumitem}

\usepackage[colorlinks=true, linkcolor=blue]{hyperref}

\theoremstyle{plain}
\newtheorem{theorem}{Theorem}[section]
\newtheorem{proposition}[theorem]{Proposition}

\newtheorem{lemma}[theorem]{Lemma}
\newtheorem{corollary}[theorem]{Corollary}
\newtheorem{conjecture}[theorem]{Conjecture}

\theoremstyle{definition}
\newtheorem{definition}[theorem]{Definition}

\theoremstyle{remark}
\newtheorem{remark}[theorem]{Remark}
\newtheorem{example}[theorem]{Example}

\newtheorem*{acks}{{\rm \bf Acknowledgments}}

\newcommand\urladdrx[1]{{\urladdr{\def~{{\tiny$\sim$}}#1}}}
\author{James Allen Fill}
\address{Department of Applied Mathematics and Statistics,
The Johns Hopkins University,
3400 N.~Charles Street,
Baltimore, MD 21218-2682 USA}
\email{jimfill@jhu.edu}
\urladdrx{http://www.ams.jhu.edu/~fill/}
\thanks{Research for J.~A.~Fill supported by
the Acheson~J.~Duncan Fund for the Advancement of Research in
Statistics.}

\author{Lachlan Ewen MacDonald}
\address{Center for Innovation in Data Engineering and Science (IDEAS), School of Engineering and Applied Science, University of Pennsylvania
220 S.~33rd Street, 
Philadelphia, PA 19104 USA}
\email{lmacdo@seas.upenn.edu}
\thanks{Research for L.~E.~MacDonald supported by the Research Collaboration on the Mathematical and Scientific Foundations of Deep Learning (NSF grant 2031985 and Simons Foundation grant 814201) and the Penn Integrates Knowledge startup fund of Ren\'{e} Vidal.}

\title[Disintegration theorem for multifunctions, with applications]{Disintegration theorem for multifunctions, \\ with applications to empirical Wasserstein distances \\ and average-case statistical bounds}
\date{July 1, 2025}

\begin{document}

\begin{abstract}
    We prove a generalisation of the disintegration theorem to the setting of multifunctions between Polish probability spaces. Whereas the classical disintegration theorem guarantees the disintegration of a probability measure along the \emph{partition} of the underlying space by the fibres of a measurable function, our theorem gives necessary and sufficient conditions for the measure to disintegrate along a \emph{cover} of the underlying space defined by the fibres of a measurable multifunction. Building on this theorem, we introduce a new statistical notion: We declare a metric Polish probability space to be \emph{asymptotically disintegrable} if $n$ i.i.d.-centred balls of decreasing radius carry a disintegration of the measure with probability tending to unity as $n\rightarrow\infty$. We give a number of both 1-dimensional and higher-dimensional examples of asymptotically disintegrable spaces with associated quantitative rates, as well as a strong counterexample. Finally, we give two applications of the notion of asymptotic disintegrability. First, we prove that asymptotically disintegrable spaces admit an easy high-probability quantification of the law of large numbers in Wasserstein space, which in all dimensions either recovers or improves upon the best known rates in some regimes, and is never any worse than existing rates by more than a factor of 2 in the exponent of~$n$, where~$n$ is the number of sample points. Second, we prove that any asymptotically disintegrable space admits a high-probability bound on the error in approximating the expectation of any Lipschitz function by its empirical average over an i.i.d.\ sample. The bound is \emph{average-case} in the sense that it depends only on the empirical average of the local Lipschitz constants of the function, rather than the global Lipschitz constant as obtained by Kantorovich--Rubinstein duality.
\end{abstract}

\maketitle

\section{Introduction}

The disintegration theorem is of fundamental importance in multiple areas of mathematics. Formally, the disintegration theorem concerns two probability spaces $(\XF, \AC, \mu)$ and $(\YF, \BC, \nu)$ and a measurable function $f:\XF \rightarrow \YF$, where $\nu = f_* \mu = \mu f^{-1}$. Provided that $(\XF,\AC,\mu)$ and $(\YF,\BC,\nu)$ are sufficiently nice (for instance, Polish spaces with their Borel $\sigma$-algebras \cite[pp.\ 78--80]{dellacherie}), the disintegration theorem guarantees the existence of a measurable family $\{\mu_y\}_{y\in\YF}$ of probability measures such that $\mu_y$ is supported on $f^{-1}\{y\}$ for $\nu$-almost every $y\in\YF$, and for which one has the disintegration formula
\begin{equation}\label{eq:disintegration}
\mu = \int\!\mu_y\,\nu(dy).
\end{equation}
The applications of this theorem are manifold, ranging from probability theory (via its use in the definition of conditional probability \cite{chang_pollard_disintegration}) to differential geometry (wherein it is implicit in the notion of ``integration over the fibre" \cite{bott_tu}). In all cases, the utility of the theorem stems from its realisation of the global measure $\mu$ as being ``patched together" from the local measures $\mu_y$ on the non-overlapping domains $f^{-1}\{y\}$.

This patching together of local objects into a single global object is a common necessity in mathematics. However, there are interesting cases in which the domains over which the local objects are defined \emph{overlap}. Consider, for instance, the problem of approximating integrals with respect to a probability measure $\mu$ on a compact metric space $\XF$ by averages with respect to an empirical measure $\hat{\mu}_n = n^{-1}\sum_{i=1}^n\delta_{x_i}$ associated to a finite sample of points $\{x_i\}_{i=1}^n$ from $\text{supp}(\mu)$. It follows from Kantorovich--Rubinstein duality that for any Lipschitz function $f$ on $\XF$ with Lipschitz constant $\|f\|_{\text{Lip}}$, one has
\begin{equation}\label{eq:kantorovichRubinstein}
\bigg|\int\!f(x)\,\mu(dx)-\int\!f(x)\,\hat{\mu}_n(dx)\bigg|\leq \|f\|_{\text{Lip}}W_1(\mu,\hat{\mu}_n),
\end{equation}
where $W_1$ denotes Wasserstein $1$-distance. 
Equation \eqref{eq:kantorovichRubinstein} has recently been utilised in the theory of deep learning \cite{buchanan_multiple_manifold,maetal} as a bound on the ``generalisation gap" between the performance $f$ of a model on the training sample $\hat{\mu}_n$ and its performance on the population $\mu$. It is of key importance in this application that the bound be applicable even to functions $f$ which depend on the training sample $\hat{\mu}$, since $\hat{\mu}$ is used in fitting the model.

However, Equation \eqref{eq:kantorovichRubinstein} is \emph{worst-case}, depending on the \emph{global} Lipschitz constant of $f$, and it is not difficult to come up with sample-dependent functions for which this bound is arbitrarily bad\footnote{Let $\XC = [0, 1]$ with Lebesgue measure $\mu$, and let $0 < x_1 < x_2 < 1$ be any two consecutive points in $\{x_i\}_{i=1}^n$. Let $L, \epsilon > 0$, and define $c:=\frac{x_1+x_2}{2}$ and $\delta:=\min\left\{\frac{x_2-x_1}{2},\sqrt{\frac{\epsilon}{L}}\right\}$. Then $f(x):=\begin{cases} L(\delta-|x-c|)&\text{ if $|x-c|\leq \delta$}\\0&\text{ otherwise}\end{cases}$ has $0<\int_{[0,1]}\!f(x)\,dx\leq \epsilon$, $\frac{1}{n}\sum_{i=1}^nf(x_i) = 0$, and $\|f\|_{\text{Lip}} = L$.}. For such functions, the bound could be improved if it depended only on an ``average" Lipschitz constant of $f$. To this end, suppose that there were some small $r>0$ for which it were possible to find a family $\{\mu_i\}_{i=1}^n$ of probability measures such that each $\mu_i$ were supported on the closed ball $B(x_i,r)$ of radius $r$ centred at $x_i$, and for which $\mu = n^{-1}\sum_{i=1}^n\mu_i$ (note that $\{B(x_i,r)\}_{i=1}^n$ is then necessarily a cover of $\text{supp}(\mu)$, which we assume without loss of generality to be all of $\XF$). Then, for any Lipschitz function~$f$, using $\|f\|_{\text{Lip},B(x_i,r)}$ to denote the local Lipschitz constant of $f$ over $B(x_i,r)$, one has
\begin{align}\label{eq:averagecase}
\bigg|\int\!f(x\,\mu(dx)-\int\!f(x)\,\hat{\mu}_n(dx)\bigg|\leq \frac{1}{n}\sum_{i=1}^n\int\!|f(x)-f(x_i)|\mu_i(dx)\leq r\cdot \frac{1}{n}\sum_{i=1}^n\|f\|_{\text{Lip},B(x_i,r)}.
\end{align}
Were it to be the case that $r \equiv r_n = O(W_1(\mu,\hat{\mu}_n))$ (an assertion we will prove in certain cases), we would have achieved our desire to obtain an ``average case" improvement on the worst-case \eqref{eq:kantorovichRubinstein}. It is evident that the condition
\begin{equation}\label{eq:partition}
    \mu = \frac{1}{n}\sum_{i=1}^n\mu_i
\end{equation}
used in the derivation of this bound closely resembles Equation \eqref{eq:disintegration}. Whereas in Equation \eqref{eq:disintegration} the $\mu_y$ were supported on the \emph{partition} of $\XF$ by the fibres $f^{-1}\{y\}$, in Equation \eqref{eq:partition} the $\mu_i$ are supported on the \emph{cover} of $\XF$ by the potentially \emph{overlapping} sets $B(x_i,r)$. That the $B(x_i,r)$ overlap implies that the map $F$ sending $x\in \XF$ to those $i\in\{1,\dots,n\}$ such that $x\in B(x_i,r)$ is a multifunction; hence, to ask for such a family of probability measures $\{\mu_i\}_{i=1}^n$ is precisely to ask for a \emph{disintegration of $\mu$ along the multifunction $F$}. Hence arises our desire to consider disintegrations of measures along multifunctions.

Having motivated disintegrations along multifunctions, let us now outline the content of the paper. In Section \ref{sec:disintegration}, we define disintegrations along multifunctions and prove the necessary and sufficient conditions for their existence in the setting of Polish spaces. Our definition of a disintegration of a probability measure $\mu$ on a space $\XF$ along a measurable multifunction $F:\XF\rightarrow \YF$ with respect to a probability measure $\nu$ on $\YF$ straightforwardly generalises that for single-valued functions: It is a measurable family $y\mapsto\mu_y$ of probability measures on $\XF$ satisfying the disintegration formula \eqref{eq:disintegration}, with $\mu_y$ supported on $F^{-1}\{y\}$ for $\nu$-almost all $y\in\YF$ (see Definition \ref{def:disintegration}).  Note that when $F$ is single-valued, our definition of disintegration reduces to the definition at the outset of this section, since it then follows that $\nu = \mu F^{-1}$: indeed, for $B \in \BC$ we have
\begin{align*}
    (\mu F^{-1})(B)
    &= \int\!\mu_y(F^{-1}(B))\,\nu(dy)
    = \int\!\mu_y(F^{-1}\{y\} \cap F^{-1}(B))\,\nu(dy) \\
    &= \int\!\mu_y(F^{-1}(\{y\} \cap B))\,\nu(dy)
    = \int_B \mu_y(F^{-1}\{y\})\,\nu(dy) \\
    &= \int_B 1\,\nu(dy) = \nu(B).
\end{align*}
We then give necessary and sufficient conditions for the existence of such disintegrations in the setting of Polish spaces in Theorem \ref{thm:existence}, which for the reader's convenience we summarise here as Theorem~\ref{thm:introexistence}.

\begin{theorem}\label{thm:introexistence}
    If $\XF$ and $\YF$ are Polish spaces and $F$ has closed graph, then $\mu$ admits a disintegration along $F$ with respect to $\nu$ if and only if $\mu(F^{-1}U)\geq\nu(U)$ for all open sets $U\subseteq\YF$.
\end{theorem}

In Section \ref{sec:disintegrability}, we consider the specific case in which $(\XF,\mu)$ carries a metric $d$ generating the Borel $\sigma$-algebra, and $\YC = [n]:=\{1,\dots,n\}$ for a natural number $n$, with $\nu = \nu_{[n]}$ the uniform measure on $[n]$. In this case, any $n$-tuple of i.i.d.\ observations $X_{[n]}:=(X_1,\dots,X_n)$ from $\mu$ on $\XF$ and radius $r>0$ defines a random multifunction $F_{r,X_{[n]}}:\XF\rightarrow \YF$ by sending $x\in \XF$ to the set of indices $i\in[n]$ such that $x\in B(X_i,r)$, and we inquire about the probability that the conditions of Theorem \ref{thm:introexistence} are satisfied by $F_{r,X_{[n]}}$. Given functions $\epsilon(n),r(n)\xrightarrow[n\to\infty]{}0$, we declare $(\XF,\mu)$ to be \emph{asymptotically disintegrable with rate $(\epsilon, r)$} if, for all sufficiently large $n$, the probability measure $\mu$ disintegrates along $F_{r(n),X_{[n]}}$ with respect to $\nu_{[n]}$ with probability at least $1-\epsilon(n)$. Proposition \ref{prop:counterexample} gives a simple example of a space 
that is \emph{not} asymptotically disintegrable (with any rate); we leave the rather nontrivial establishment of \emph{positive} examples of asymptotic disintegrability until the final Section~\ref{sec:examples}. We conclude Section \ref{sec:disintegrability} by motivating these (difficult) examples with (easy) applications (summarised here for the reader's convenience as Theorem~\ref{thm:applications}) that follow from the abstract notion of asymptotic disintegrability.

\begin{theorem} \label{thm:applications}
    Suppose that $(\XF,\mu)$ is asymptotically disintegrable with rate $(\epsilon,r)$. For $n$ sufficiently large, let $X_1,\dots,X_n$ be i.i.d.\ samples from $\mu$ with empirical measure $\hat{\mu}:=\frac{1}{n}\sum_{i=1}^n\delta_{X_i}$. Then:
    \begin{enumerate}
        \item[{\rm (i)}] For any $p\in[1,\infty)$, with probability at least $1-\epsilon(n)$, one has
        \begin{equation}\label{eq:wbound}
        W_p(\mu,\hat{\mu}_n)\leq r(n),
        \end{equation}
        where $W_p$ is the Wasserstein $p$-distance.  (See Proposition \ref{prop:wasserstein} for details.)
        \item[{\rm (ii)}] With probability at least $1-\epsilon(n)$, one has
        \begin{equation}\label{eq:avgbound}
        \bigg|\int\!f(x)\,\mu(dx)-\int\!f(x)\,\hat{\mu}_n(dx)\bigg|\leq r(n)\cdot\frac{1}{n}\sum_{i=1}^n\|f\|_{\text{Lip},B(X_i,r(n))}
        \end{equation}
        for any Lipschitz function $f:\XF\rightarrow\RB$.  (See Proposition \ref{prop:averagecase} for details.)
    \end{enumerate}
\end{theorem}

 In Section \ref{sec:examples}, we give results that we found by far the most challenging to prove, namely, Theorems \ref{T:path main}, \ref{T:circle}, \ref{T:graph main}, and \ref{T:dd main}, which prove that certain simple spaces \emph{are} asymptotically disintegrable. Theorem \ref{thm:introexamples} below summarises these results.

\begin{theorem}\label{thm:introexamples}
    For any $\alpha\in(0,\infty)$:
    \begin{enumerate}    
        \item[{\rm (i)}] The interval $[0,1]$ with its Euclidean metric and absolutely continuous measure with density bounded away from both~$0$ and~$\infty$ is asymptotically disintegrable with rate 
        \begin{equation} \label{big-oh}
        \mbox{$\left(O\left(n^{-\alpha} \right),O\!\left(\sqrt{\frac{\ln n}{n}}\right)\right)$}. 
        \end{equation}
        \item[{\rm (ii)}] The circle in $\RB^2$ with normalised geodesic metric and absolutely continuous measure with density bounded away from both~$0$ and~$\infty$ is asymptotically disintegrable with rate~\eqref{big-oh}.
        \item[{\rm (iii)}] Any finite connected simple graph $(V,E)$ embedded in $\RB^2$ with appropriate graph metric and absolutely continuous measure with with density bounded away from both~$0$ and~$\infty$ is asymptotically disintegrable with rate~\eqref{big-oh}.
        \item[{\rm (iv)}] For any natural number $D\geq 2$, the unit cube $[0,1]^D$ with its $\ell^{\infty}$ metric and any absolutely continuous measure with with density bounded away from both~$0$ and~$\infty$ is absolutely disintegrable with rate $\left(O(n^{-\alpha}),O\!\left((\frac{\ln n}{n})^{1 / (2D)}\right)\right)$.
    \end{enumerate}
\end{theorem}

As will be made clear in Section~\ref{sec:wdre}, our bound $r(n)$ in~\eqref{eq:wbound} for $W_p(\mu,\hat{\mu}_n)$ either matches or improves the best high-probability rates in the literature for our unit cube examples for large values of~$p$, and is in general worse than the best existing rates by no more than a factor of 2 in the exponent on $n$. However, it is clear that asymptotic disintegrability is not \emph{necessary} to achieve good high probability bounds, since the same Wasserstein distance rate for $p = 1$ we obtain for $\XF = [0, 1]$ can be achieved using existing techniques \cite[first display on p.~2305]{boissard} even for the non-asymptotically disintegrable 
example considered in Proposition \ref{prop:counterexample}\footnote{Connectedness \emph{is}, however, \emph{necessary} to obtain the best rate for $\EB W_p(\mu,\hat{\mu}_n)$ when $p>1$, see \cite[Theorem 5.6]{wasserstein1d}.}. We conjecture, however, that asymptotic disintegrability \emph{is} necessary for the average-case bound \eqref{eq:avgbound}; while we cannot yet prove this conjecture, Proposition \ref{prop:averagecasecounterexample} demonstrates suggestively that our
counterexample to asymptotic disintegrability in Proposition \ref{prop:counterexample} \emph{also} provides a counterexample to \eqref{eq:avgbound}, and for essentially the same reason.

We conjecture that many more examples are also asymptotically disintegrable, for instance, compact connected Riemannian manifolds with volume measure.  We also conjecture that our rate $O\!\left((\frac{\ln n}{n})^{1 / (2D)}\right)$ for $[0,1]^D$ in the case $D\geq 2$ can be tightened to $O\!\left((\frac{\ln n}{n})^{1 / D}\right)$, thus enabling new proofs of the best known rates $r(n)$ (up to logarithmic factors) for the high-probability Wasserstein distance bounds \eqref{eq:wbound} (see~\cite{weedbach, lei}). 
These conjectures are left as open problems.

\section{The disintegration theorem for multifunctions}\label{sec:disintegration}

Recall that a \emph{multifunction} $F:\XF\rightarrow\YF$ between sets $\XF$ and $\YF$ is an assignment to each $x\in\XF$ of a (possibly empty) subset $F(x)\subseteq \YF$. By a \emph{Polish probability space} we mean a triple $(\XF,\AC,\mu)$, where $\XF$ is a Polish space, $\AC$ its Borel $\sigma$-algebra, and $\mu$ a Borel probability measure. In this section, we prove necessary and sufficient conditions for disintegrability of a measure along a multifunction in the context of Polish probability spaces.  We begin by defining the disintegration of a probability measure along a multifunction between probability spaces as follows.

\begin{definition}\label{def:disintegration}
    Let $(\XF,\AC,\mu)$ and $(\YF,\BC,\nu)$ be probability spaces, and let $F:\XF\rightarrow \YF$ be a multifunction which is \emph{measurable} in the sense that $F^{-1}B := \{x \in \XF:F(x) \cap B \neq \varnothing\} \in \AC$ for all $B\in\BC$. A \emph{disintegration of $\mu$ along $F$ with respect to $\nu$} is a map $y\mapsto\mu_y$ from $\YF$ to the space of Borel probability measures on $\XF$ such that:
    \begin{enumerate}
        \item[{\rm (i)}] $y\mapsto\mu_y$ is \emph{measurable}, in the sense that $y\mapsto\mu_y(A)$ is $\BC$-measurable for any $A\in\AC$;
        \item[{\rm (ii)}] $\mu_y$ is supported on the fibre $F^{-1}\{y\} = \{x\in\XF:y\in F(x)\}$ for $\nu$-almost every $y\in \YF$; and
        \item[{\rm (iii)}] one has the disintegration formula $\mu = \int\!\mu_y\,\nu(dy)$.
    \end{enumerate}
\end{definition}

Our multifunction disintegration theorem (Theorem \ref{thm:existence}) is a consequence of (i)~Strassen's coupling theorem on the existence of probability measures with given marginals and (ii)~the classical disintegration theorem. We recall these two theorems in the setting relevant to our needs next.

\begin{theorem}[Strassen's theorem for Polish probability spaces]
    Let $(\XF, \AC, \mu)$ and $(\YF, \BC, \nu)$ be Polish probability spaces, and let $\omega$ be a nonempty, closed subset of $\XF \times \YF$.  Then there exists a probability measure $\lambda$ on $\XF\times \YF$ (with its Borel $\sigma$-algebra) with marginals $\mu$ and $\nu$ such that $\lambda(\omega) = 1$ if and only if
    \begin{equation}\label{eq:strassens}
    \mu\big(\pi_{\XF}(\omega\cap(\XF\times U))\big)\geq \nu(U)
    \end{equation}
    for all open sets $U\subseteq \YF$, where $\pi_{\XF}:\XF\times \YF\rightarrow \XF$ is the projection of $\XF \times \YF$ onto $\XF$.
\end{theorem}

\begin{proof}
    This is an immediate consequence of \cite[Theorem 11]{strassen}, letting $\epsilon$ in that theorem vanish.
\end{proof}

\begin{theorem}[Disintegration theorem for Polish spaces]
\label{thm:classical}
    Let $(\XF,\AC)$ and $(\YF,\BC)$ be Polish spaces, assumed to be equipped with their Borel $\sigma$-algebras $\AC$ and $\BC$, respectively, and let $f:\XF\rightarrow\YF$ be a Borel-measurable map. Given a Borel probability measure $\mu$ on $\XF$ and $\nu:=f_*\mu$ on $\YF$, there exists a measurable family $y\mapsto\mu_y$ of probability measures on $\YF$ such that $\mu = \int\!\mu_y\,\nu(dy)$ and such that $\mu_y$ is supported on $f^{-1}\{y\}$ for $\nu$-almost all $y\in\YF$.
\end{theorem}

\begin{proof}
This is a consequence of \cite[pp.\ 78--80]{dellacherie}. Specifically, since $\XF$ is a Polish space and $\mu$ is a finite Borel measure, it follows from Ulam's theorem \cite[Theorem 7.1.4]{dudleybook} that $\mu$ is tight. Then \cite[Theorem 72, p.~79]{dellacherie} implies the existence of a measurable family $y\mapsto\mu_y$ such that $\mu = \int\!\mu_y\,\nu(dy)$. Moreover, by \cite[(73.1), p.~80]{dellacherie} and the discussion that follows, the fact that $(\YF,\BC)$ is a Polish space with Borel $\sigma$-algebra implies that $\mu_y$ is supported on $f^{-1}\{y\}$ for $\nu$-almost all $y\in\YF$.
\end{proof}

We now come to our disintegration theorem for multifunctions. Recall that the \emph{graph} of a multifunction $F:\XF\rightarrow\YF$ is the set
\[
\text{Graph}(F):=\{(x,y)\in \XF\times\YF:y\in F(x)\}.
\]

\begin{theorem}[Disintegration theorem for multifunctions]\label{thm:existence}
    Let $(\XF,\AC,\mu)$ and $(\YF,\BC,\nu)$ be Polish probability spaces, and let $F:\XF\rightarrow \YF$ be a measurable multifunction with graph $\text{\rm Graph}(F)$ closed in the product topology on $\XF\times\YF$. Then there exists a disintegration of $\mu$ along $F$ with respect to $\nu$ if and only if
    \begin{equation}\label{eq:ours}
    \mu\big(F^{-1}U\big)\geq\nu(U)
    \end{equation}
    for all open sets $U \subseteq \YF$.
\end{theorem}

\begin{proof}
    Note that since $F$ is measurable, $F^{-1}U$ is Borel for any open set $U\subseteq\YF$, so that the left side of \eqref{eq:ours} is well defined. Necessity of \eqref{eq:ours} is then easy. Specifically, suppose that $\{\mu_y\}_{y\in\YF}$ is a disintegration of $\mu$ along $F$ with respect to $\nu$, and let $U \subseteq \YF$ be open. Then
    \begin{align*}
        \mu\big(F^{-1}U\big) 
        &= \int\!\mu_y\big(F^{-1}U\big)\,\nu(dy) \\ 
        &\geq \int_U\mu_y\big(F^{-1}U\big)\,\nu(dy) \\
        &=\int_U1\,\nu(dy) = \nu(U),
    \end{align*}
    where the first equality follows from the disintegration formula and the second equality follows from the fact that $\mu_y$ is supported on $F^{-1}\{y\}$ for $\nu$-almost all $y\in \YF$. 
    
    To see that \eqref{eq:ours} is sufficient, let $U\subseteq \YF$ be open, and note that
    \[
    \pi_{\XF}\big(\text{Graph}(F)\cap(\XF\times U)\big) = \{x\in \XF:F(x)\cap U\neq\varnothing\} = F^{-1}U.
    \]
    Thus our condition \eqref{eq:ours} is precisely Strassen's condition \eqref{eq:strassens} with $\omega = \text{Graph}(F)$. Consequently, assuming \eqref{eq:ours} enables us to apply Strassen's theorem to give a probability measure $\lambda$ on $\XF\times \YF$ with marginals $\mu$ and $\nu$ such that $\lambda(\text{Graph}(F)) = 1$. Apply the classical disintegration theorem (Theorem \ref{thm:classical}) to the measurable map $\pi_{\YF}:\XF\times\YF\rightarrow\YF$ to obtain a measurable family $\{\lambda_y\}_{y\in\YF}$ of probability measures on $\XF\times\YF$ such that $\lambda = \int\!\lambda_y\,\nu(dy)$ and $\lambda_y$ is supported on $\pi_{\YF}^{-1}\{y\}$ for $\nu$-almost every $y \in \YF$. Set $\mu_y:=(\pi_{\XF})_*\lambda_y$ for each $y\in\YF$. Then measurability of the family $y\mapsto\mu_y$ follows from measurability of $y\mapsto\lambda_y$ and measurability of $\pi_{\XF}$. Moreover, the disintegration formula for $\mu$ in terms of $\{\mu_y\}_{y\in\YF}$ follows from the disintegration formula for $\lambda$ in terms of $\{\lambda_y\}_{y\in\YF}$:
    \[
    \mu 
    = (\pi_{\XF})_*\lambda 
    = (\pi_{\XF})_*\int\!\lambda_y\,\nu(dy) 
    = \int\!(\pi_{\XF})_*\lambda_y\,\nu(dy) 
    = \int\!\mu_y\,\nu(dy).
    \]
    Finally, observe that since $\lambda(\text{Graph}(F)) = 1$, one must have $\lambda_y(F^{-1}\{y\}\times\{y\}) = 1$ for $\nu$-almost every $y \in \YF$, so that for such $y$ we have
    \[
    \mu_y(F^{-1}\{y\}) 
    = \lambda_y\big(\pi_{\XF}^{-1}F^{-1}\{y\}\big) 
    = \lambda_y\big(F^{-1}\{y\}\times\YF\big) 
\geq \lambda_y\big(F^{-1}\{y\}\times\{y\}\big) 
    = 1.
    \]
    Thus $y\mapsto\mu_y$ is a disintegration of $\mu$ along $F$ with respect to $\nu$.
\end{proof}

\section{Applications}\label{sec:disintegrability}

The purpose of this section is to explicate applications of Theorem \ref{thm:existence} to statistics, by considering the special case where the multifunction is defined on a metric space by membership of balls with random centres. Specifically, we consider the following (from what we can determine, non-standard) specification of a metric Polish probability space.

\begin{definition}
    A \emph{metric Polish space} is a couple $(\XF,d)$, where $\XF$ is a Polish space and $d$ is a metric on $\XF$ which generates the topology of $\XF$. (\emph{We do not insist that $d$ be complete.}) A \emph{metric Polish probability space} is a triple $(\XF,d,\mu)$, where $(\XF,d)$ is a metric Polish space and $\mu$ is a Borel probability measure on $\XF$.
\end{definition}

Let $(\XF,d,\mu)$ be a metric Polish space. Given an $n$-tuple $X_{[n]}:=(X_1,\dots,X_n)$ of i.i.d.\ random variables with distribution $\mu$ and a radius $r>0$, the closed balls $\{B(X_i,r)\}_{i=1}^n$ of radius $r$ with random centres $\{X_i\}_{i=1}^n$ define a random multifunction $F_{r,X_{[n]}}:\XF \rightarrow [n]$ according to the formula
\begin{equation}\label{eq:Fn}
F_{r, X_{[n]}}(x):=\{i\in[n]:x\in B(X_i,r)\}.
\end{equation}
A disintegration of $\mu$ along $F_{r,X_{[n]}}$ with respect to  uniform probability $\nu_{[n]}$ on $[n]$ is then a family $\{\mu_i\}_{i=1}^n$ of probability measures such that $\text{supp}(\mu_i)\subseteq B(X_i,r)$ for all $i$ and for which $\mu = \frac{1}{n}\sum_{i=1}^n\mu_i$. The necessary and sufficient conditions of Theorem \ref{thm:existence} in this setting reduce simply to
\begin{equation}\label{nasc}
\mu\bigg(\bigcup_{i\in I}B(X_i,r)\bigg)\geq \frac{|I|}{n}\text{\ for all $I\in 2^{[n]}$}.
\end{equation}
Clearly \eqref{nasc} holds with probability 1 whenever $r$ is equal to the diameter of $(\XF,d)$. Ideally, \eqref{nasc} would also hold with high probability for smaller values of $r$. We pin down this notion in the following two definitions. Throughout the sequel we will use $\NB$ to denote the strictly positive integers and $\RB_{\geq0}$ to denote the non-negative real numbers.

\begin{definition}\label{def:disint1}
    Given $n\in\NB$, $\epsilon > 0$, and $r > 0$, we declare a metric Polish probability space $(\XF,d,\mu)$ to be $(n,\epsilon,r)$\emph{-disintegrable} if, whenever $X_1,\dots,X_n$ are i.i.d.\ with distribution $\mu$ on a probability space $(\Psi,\FC,\PB)$, there exists a disintegration of $\mu$ along $F_{r, X_{[n]}}$ with respect to the uniform measure on $[n]$ with $\PB$-probability at least $1-\epsilon$.
\end{definition}

\begin{remark} \label{R:monotonicity}
    If a metric Polish probability space is $(n, \epsilon, r)$-disintegrable, $\epsilon' \geq \epsilon$ and $r' \geq r$, then [recall~\eqref{nasc}] the space is $(n, \epsilon', r')$-disintegrable.
\end{remark}

The following stronger notion of disintegrability has $(\epsilon,r)$ scaling to zero as $n\rightarrow\infty$.

\begin{definition}\label{def:disint2}
    Given functions $\epsilon,r:\NB\rightarrow\RB_{\geq0}$ tending to zero as $n\rightarrow\infty$, we declare a metric Polish probability space $(\XF,d,\mu)$ to be \emph{asymptotically disintegrable with rate $(\epsilon,r)$} if for all $n$ sufficiently large, $(\XF,d,\mu)$ is $(n,\epsilon(n),r(n))$-disintegrable.
\end{definition}

It might be thought that any reasonably nice space is necessarily asymptotically disintegrable. This is not the case, as the following counterexample shows.

\begin{proposition}\label{prop:counterexample}
    For fixed irrational $q \in (0, 1)$, consider the metric Polish probability space 
    \[
    \XF := [-1, - (1 - q)] \cup [0, 1 - q],
    \]
    equipped with its Euclidean metric $d$ and uniform probability $\mu$. Then $(\XF,d,\mu)$ is not $(n,\epsilon,r)$-disintegrable for any $n \in\NB$, $r < 1 - q$, and $\epsilon < 1$. In particular, $(\XF,d,\mu)$ is not asympotically distintegrable for any rate.
\end{proposition}

\begin{proof}
        Fix $n\in\NB$, $r<1-q$, and $\epsilon<1$, and suppose for a contradiction that $(\XF,d,\mu)$ is $(n,\epsilon,r)$-disintegrable. Denote $\XF_{-} := [-1, - (1 - q)]$ and $\XF_{+} := [0, 1 - q]$.  Observe that if $x \in \XF_{-}$, then $B(x, r) \subseteq \XF_{-}$; similarly, if $x \in \XF_{+}$, then $B(x, r) \subseteq \XF_{+}$.
    Let $X_1,X_2,\dots,X_n$ be i.i.d.\ with distribution $\mu$ on a probability space $(\Psi,\FC,\PB)$. Define $J_n$ to be the random subset $\{i \in [n]:X_i \in \XF_{-}\}$ of $[n]$.  By hypothesis, there is an event of $\PB$-probability at least $1-\epsilon>0$ on which the inequality in~\eqref{nasc} holds both for $I = J_n$ and for $I = [n] \setminus J_n$; but then we must have
    \[
    \mu\bigg(\bigcup_{i \in J_n} B(X_i,r) \bigg) = \frac{|J_n|}{n} = \mu(\XF_{-}) = q,
    \]
    contradicting the irrationality of~$q$.
\end{proof}

In Section \ref{sec:examples}, we will provide a number of \emph{positive} examples of asymptotically disintegrable metric Polish probability spaces. It will become apparent that \emph{proving} the asymptotic disintegrability of a given metric Polish probability space is a non-trivial matter. To motivate the work required to establish this property, therefore, let us first describe some applications of this notion. The first application (Section \ref{sec:Wdb}) is to bounding the Wasserstein distance between a measure and an empirical measure constructed therefrom by i.i.d.\ samples; the second (Section \ref{sec:Acab}) is to the provision of ``average-case bounds" on the approximation of integrals of Lipschitz functions by empirical averages.

\subsection{Wasserstein distance bounds}\label{sec:Wdb}

If $(\XF,d)$ is a metric Polish space, $p\in[1,\infty)$, and $\mu$ and $\nu$ are Borel probability measures on $\XF$, denote by $\Lambda(\mu,\nu)$ the set of all Borel probability measures on $\XF\times\XF$ with marginals $\mu$ and $\nu$. Recall that the \emph{Wasserstein $p$-distance} $W_p(\mu, \nu)$ between such $\mu$ and $\nu$ is then defined by
\[
W_p(\mu,\nu):=\inf_{\lambda\in\Lambda(\mu,\nu)}\bigg(\int\!d(x,x')^p\,\lambda(dx,dx')\bigg)^{1 / p}.
\]
Given $n$ i.i.d.\ random variables $(X_1,\dots,X_n)$ on a probability space $(\Psi,\FC,\PB)$ with distribution $\mu$, use $\hat{\mu}_{n}$ to denote the random empirical measure $\frac{1}{n}\sum_{i=1}^n\delta_{X_i}$.

\begin{proposition}\label{prop:wasserstein}
    Suppose that a metric Polish probability space $(\XF,d,\mu)$ is asymptotically disintegrable with rate $(\epsilon,r)$ and that $X_1,X_2,\dots$ are i.i.d.\ with distribution $\mu$ on a probability space $(\Psi,\FC,\PB)$. Then, for all $n$ sufficiently large, one has for each fixed $p \in [1, \infty)$ that
    \[
    W_p(\mu,\hat{\mu}_n)\leq r(n)
    \]
    with $\PB$-probability at least $1-\epsilon(n)$.
\end{proposition}

\begin{proof}
    Since $(\XF,d,\mu)$ is asymptotically disintegrable with rate $(\epsilon,r)$, for all $n\in\NB$ sufficiently large there is an event of $\PB$-probability at least $1-\epsilon(n)$ such that $\mu$ admits a disintegration $\{\mu_i\}_{i=1}^n$ along the multifunction $F_{r(n),X_{[n]}}:\XF\rightarrow [n]$ defined by the balls $B(X_i,r(n))$. Conditioning on this event, the disintegration formula $\mu = \frac{1}{n}\sum_{i=1}^n\mu_i$ implies that $\lambda_n:=\frac{1}{n}\sum_{i=1}^n (\mu_i\times\delta_{X_i})$ is a coupling between $\mu$ and $\hat{\mu}_n$; thus
    \begin{align*}
    W_p(\mu,\hat{\mu}_n)
    &\leq \bigg(\int\!d(x,x')^p\,\lambda_n(dx,dx')\bigg)^{\frac{1}{p}} 
    = \bigg(\frac{1}{n}\sum_{i=1}^n\int\!d(x,X_i)^p\,\mu_i(dx)\bigg)^{\frac{1}{p}}\\
    &\leq \bigg(\frac{1}{n}\sum_{i=1}^n\int\!r(n)^p\,\mu_i(dx)\bigg)^{\frac{1}{p}} 
    = r(n),
    \end{align*}
    where the second inequality follows from the fact that $\mu_i$ is supported on $B(X_i,r(n))$ for all $i$, and the final identity from the fact that each $\mu_i$ has mass $1$.
\end{proof}

Note that, unlike existing high-probability bounds on $W_p(\mu, \hat{\mu}_n)$, the bound $r(n)$ in 
Proposition \ref{prop:wasserstein} does not depend on $p \in [1, \infty)$.
Combining Proposition \ref{prop:wasserstein} with Theorem \ref{T:path main}, \ref{T:circle}, \ref{T:graph main},  or \ref{T:dd main} or Corollary~\ref{C:non d main} thus gives a new proof of a high-probability quantitative law of large numbers in Wasserstein space for certain spaces $(\XF, d, \mu)$. For the $1$-dimensional examples we consider, our rates at least match the best known rates, while for the higher-dimensional examples we consider our rates are never worse than by a factor of 2 in the exponent. For sufficiently large $p$, our bounds are always better than the best known bounds due to the uniformity of our bounds in $p$.  We will discuss this more precisely for the examples of Theorem~\ref{T:path main} and Corollary~\ref{C:non d main} in Section~\ref{sec:wdre}.

\subsection{Average-case approximation bounds}\label{sec:Acab}

Our second application is the primary motivation for this work:\ the provision of \emph{average-case} bounds on the error in approximating the expectations of Lipschitz functions by empirical averages thereof. Given a metric space $(\XF,d)$, Lipschitz function $f:\XF\rightarrow\RB$, and $x\in\XF$, we denote by $\|f\|_{\text{Lip},B(x,r)}$ the local Lipschitz constant of $f$ over the closed ball $B(x,r)$.

\begin{definition}\label{def:averagecase}
    Let $(\XF,d,\mu)$ be a metric Polish probability space. Given $n\in\NB$, $\epsilon > 0$ and $r>0$, we say that $(\XF,d,\mu)$ is \emph{$(n,\epsilon,r)$-average-case approximable} if, whenever $X_1,\dots,X_n$ are i.i.d.\ with distribution $\mu$ on a probability space $(\Psi,\FC,\PB)$, there is an event of $\PB$-probability at least $1-\epsilon$ on which one has
    \[
    \bigg|\int\!f(x)\,\mu(dx) - \frac{1}{n}\sum_{i=1}^nf(X_i)\bigg|\leq r\cdot\frac{1}{n}\sum_{i=1}^n\|f\|_{\text{Lip},B(X_i,r)}
    \]
    for all Lipschitz functions $f:\XF\rightarrow\RB$. Given $\epsilon,r:\NB\rightarrow\RB_{\geq0}$ which vanish as $n\rightarrow\infty$, we say that $(\XF,d,\mu)$ is \emph{asymptotically average-case approximable with rate $(\epsilon,r)$} if, for all $n$ sufficiently large, $(\XF,d,\mu)$ is $(n,\epsilon(n),r(n))$-average-case approximable.
\end{definition}

``Average-case" here is used to refer to the fact that the bound depends only on the \emph{average} of the \emph{local} Lipschitz constants of $f$. We contrast Definition \ref{def:averagecase} with the \emph{worst-case} approximation bound that is easily derivable from Kantorovich--Rubinstein duality, namely
\[
\bigg|\int\!f(x)\,\mu(dx)-\frac{1}{n}\sum_{i=1}^nf(X_i)\bigg|\leq \|f\|_{\text{Lip}} W_1(\mu,\hat{\mu}_n)\text{ for all Lipschitz $f:\XF\rightarrow\RB$,}
\]
depending on the \emph{global} Lipschitz constant $\|f\|_{\text{Lip}}$ of $f$. Our next proposition demonstrates that asymptotic disintegrability is sufficient for asymptotic average-case approximability.

\begin{proposition}\label{prop:averagecase}
    Let $(\XF,d,\mu)$ is a metric Polish probability space which is asymptotically disintegrable with rate $(\epsilon,r)$. Then $(\XF,d,\mu)$ is asymptotically average-case approximable with rate $(\epsilon,r)$.
\end{proposition}

\begin{proof}
    Let $X_1,X_2,\dots$ be i.i.d.\ with distribution $\mu$ on a probability space $(\Psi,\FC,\PB)$. Since $(\XF,d,\mu)$ is asymptotically $(\epsilon,r)$-disintegrable, for all $n$ sufficiently large there is an event of $\PB$-probability at least $1-\epsilon(n)$ on which $\mu$ admits a disintegration $\{\mu_i\}_{i=1}^n$ along the multifunction $F_{r(n),X_{[n]}}:\XF\rightarrow [n]$ defined by the balls $B(X_i,r(n))$ [recall \eqref{eq:Fn}]. Conditioning on this event, one has
    \begin{align*}
        \bigg|\int\!f(x)\,\mu(dx)-\frac{1}{n}\sum_{i=1}^nf(X_i)\bigg|& = \bigg|\frac{1}{n}\sum_{i=1}^n\int (f(x)-f(X_i))\,\mu_i(dx)\bigg|\\
        &\leq \frac{1}{n}\sum_{i=1}^n\int\!|f(x)-f(X_i)|\,\mu_i(dx)\\
        &\leq r(n)\cdot\frac{1}{n}\sum_{i=1}^n\|f\|_{\text{Lip},B(X_i,r(n))},
    \end{align*}
    where the equality follows from the disintegration formula and the second inequality from the fact that $\mu_i$ is a probability measure supported on $B(X_i,r(n))$.
\end{proof}

A natural question is whether asymptotic disintegrability is \emph{necessary} for asymptotic average-case approximability. We cannot yet prove that this is the case, but we \emph{can} prove that the counterexample to asymptotic disintegrability in Proposition \ref{prop:counterexample} is also a counterexample to asymptotic average-case approximability, and for essentially the same reason. 
The formal similarity of these two propositions and their proofs motivates our conjecture that asymptotic disintegrability is indeed necessary for asymptotic average-case approximability.

\begin{proposition}\label{prop:averagecasecounterexample}
    Consider the example $(\XF, d, \mu)$ in Proposition \ref{prop:counterexample}. Then $(\XF,d,\mu)$ is not $(n,\epsilon,r)$-average case approximable for any $n\in\NB$, $r < 1-q$ and $\epsilon<1$. In particular, $(\XF,d,\mu)$ is not asymptotically average-case approximable for any rate.
\end{proposition}

\begin{proof}
    The proof idea is much the same as that of Proposition \ref{prop:counterexample}. Fixing $n\in\NB$, $r<1-q$, $\epsilon <1$ and supposing that $(\XF,d,\mu)$ were $(n,\epsilon,r)$-average case approximable, consider i.i.d.\  random variables $X_1,\dots,X_n$ with distribution $\mu$ on a probability space $(\Psi,\FC,\PB)$. Consider the Lipschitz function $f_n:=n {\bf 1}_{\XF_{-}}:\XF\rightarrow\RB$, i.e.,\ $n$ times the indicator function of $\XF_{-}$; the local Lipschitz constant of $f_n$ on each of the balls $B(X_i,r)$ necessarily vanishes. Then, by hypothesis, with $J_n := \{i \in [n]:X_i \in \XF_{-}\}$, there is an event of $\PB$-probability at least $1-\epsilon>0$ on which
    \[
    qn = \int\!f_n(x)\,\mu(dx)=\frac{1}{n}\sum_{i=1}^nf_n(X_i) = |J_n|,
    \]
    contradicting the irrationality of $q$.
\end{proof}

In Section \ref{sec:examples}, we demonstrate that a number of connected spaces are  asymptotically disintegrable. 

\section{Examples}\label{sec:examples}

In this section, we prove that a number of simple spaces are asymptotically disintegrable and give associated rates.

\subsection{The interval and the circle}\label{sec:interval}

In this subsection, we prove that the interval (Theorem~\ref{T:path main}) and the circle (Theorem~\ref{T:circle}) are asymptotically disintegrable with rate $\big(O(n^{-\alpha}),O\big(\sqrt{\ln n/n})\big)$ for any~$\alpha$, but we begin with some general considerations.

\begin{definition}
Let $(\Omega, \AC, \mu)$ be a probability space with a given cover by measurable sets $B_1, \ldots, B_n$, and let $I \in 2^{[n]}$.  We will say that~$I$ is \emph{reducible} if there exists a nonempty proper subset~$J$ of~$I$ such that
\begin{equation}
\label{reducible_interval}
\mu\!\left( \left( \bigcup_{i \in J} B_i \right) \cap \left( \bigcup_{i \in I \setminus J} B_i \right) \right) = 0.
\end{equation}
We will say that~$I$ is \emph{irreducible} if it is not reducible.
\end{definition}

The following proposition allows us to reduce the work in checking~\eqref{nasc}.

\begin{proposition}
\label{P:irreducible_interval}
Given a probability space $(\Omega, \AC, \mu)$ with a cover by measurable sets $B_1, \ldots, B_n$ and a probability measure~$\nu$ on $[n]$, \eqref{nasc} holds if and only if the inequalities therein are satisfied for every irreducible $I \in 2^{[n]}$.
\end{proposition}

\begin{proof}
This is easy.  The ``only if'' assertion is trivial.  To prove the ``if'' assertion, we use induction on $|I|$, with the cases $|I| = 0$ and $|I| = 1$ being trivial.  We proceed to the induction step with $|I| \geq 2$.  If~$I$ is irreducible, then
\[
\mu\!\left( \bigcup_{i \in I} B_i \right) \geq \sum_{i \in I} \nu_i 
\]
holds by assumption.  So suppose that~$I$ is reducible.  Let~$J$ be a nonempty proper subset of~$I$ satisfying~\eqref{reducible_interval}.
Then
\begin{align*}
\mu\!\left( \bigcup_{i \in I} B_i \right) 
&= \mu\!\left( \bigcup_{i \in J} B_i \right) + \mu\!\left( \bigcup_{i \in I \setminus J} B_i \right) 
- \mu\!\left( \left( \bigcup_{i \in J} B_i \right) \cap \left( \bigcup_{i \in I \setminus J} B_i \right) \right) \\ 
&= \mu\!\left( \bigcup_{i \in J} B_i \right) + \mu\!\left( \bigcup_{i \in I \setminus J} B_i \right) \\
&\geq \sum_{i \in J} \nu_i + \sum_{i \in I \setminus J} \nu_i \\
&= \sum_{i \in I} \nu_i.
\end{align*}
where the second equality follows from~\eqref{reducible_interval} and the inequality follows by the induction hypothesis.
\end{proof}

Consider now the metric Polish probability space $(\XC,d,\mu)$, with $\XC = [0,1]$, $d$ its Euclidean metric, and $\mu$ any Borel probability measure (later, we will need to insist that $\mu$ satisfy some additional conditions). Suppose that~$n$ balls $B(x_i, r)$, with $x_i \in \XC$ and (necessarily) $r > 0$, cover~$\XC$. Write $y_1, \ldots, y_n$ for $x_1, \ldots, x_n$ rearranged into nondecreasing order, and let $B_i = B(y_i, r)$.

\begin{proposition}
\label{P:I to L_interval}
Let~$\mu$ be arbitrarily specified.  If 
a set $I \in 2^{[n]}$ is irreducible, then there exists $L \supseteq I$ such that~$L$ is connected in the usual path-graph (call it 
$\PC_n$) on $[n]$ and
\begin{equation}
\label{I to L_interval}
\cup_{i \in I} B_i = \cup_{i \in L} B_i.
\end{equation}
\end{proposition}

\begin{proof}
If~$I$ is connected in $\PC_n$, we can choose $L = I$. 
Otherwise, there exist $i_1$ and $i_2$ in~$I$ with $i_2 \geq i_1 + 2$ such that $i_1 + 1 \notin I$ and $i_2 - 1 \notin I$.
We will prove by contradiction that either $B_{i_1 + 1} \subseteq \cup_{i \in I} B_i$ or $B_{i_2 - 1} \subseteq \cup_{i \in I} B_i$.
Indeed, if $B_{i_1 + 1} \not\subseteq \cup_{i \in I} B_i$, then $y_{i_1 + 1} - y_{i_1} > r$.  If also 
$B_{i_2 - 1} \not\subseteq \cup_{i \in I} B_i$, then (by the same reasoning) $y_{i_2} - y_{i_2 - 1} > r$.

Now let~$J$ be the interval $[i_1 + 1, i_2]$ in $\PC_n$, and note that $J$ is a nonempty proper subset of $[n]$.  We then claim that~\eqref{reducible_interval} holds, contradicting our assumption that~$I$ is irreducible.  Indeed, 
\[
\left( \bigcup_{i \in J} B_i \right) \cap \left( \bigcup_{i \in I \setminus J} B_i \right) = \varnothing,
\]
because (by construction) the distance from $y_{i_1}$ to $y_{i_2}$ (strictly) exceeds $2 r$. 

So we may now suppose that there exists $i_0 \in I$ with a neighbour $k \in \PC_n$ such that $k \notin I$ and 
$B_k \subseteq \cup_{i \in I} B_i$.  If we enlarge $I$ to get~$\widetilde{I}$ by annexing the element~$k$, then
\[
\cup_{i \in I} B_i = \cup_{i \in \widetilde{I}} B_i.
\]  
If~$\widetilde{I}$ is connected in~$\PC_n$, then we can choose $L = \widetilde{I}$.  If not, we continue annexing elements until $L \supseteq I$ is connected in~$\PC_n$ and~\eqref{I to L_interval} holds.
\end{proof}

\begin{corollary}
\label{C:connected_interval}
Let~$\mu$ be arbitrarily specified, and let~$\nu$ be a probability measure on $[n]$.  If~\eqref{nasc} holds for every connected set~$I$ in the path-graph $\PC_n$ on $[n]$, then it holds for every $I \subseteq [n]$.
\end{corollary}

\begin{proof}
By Proposition \ref{P:irreducible_interval}, it is sufficient that~\eqref{nasc} hold for every irreducible~$I$.  But in that case we can apply 
Proposition \ref{P:I to L_interval} to obtain~$L \supseteq I$ that is connected in $\PC_n$ and satisfies~\eqref{I to L_interval}, and we then observe
\[
\mu\!\left( \bigcup_{i \in I} B_i \right) 
= \mu\!\left( \bigcup_{i \in L} B_i \right) 
\geq \sum_{i \in L} \nu_i \geq \sum_{i \in I} \nu_i,
\] 
as desired.
\end{proof}

Now we come to our disintegrability theorem for the interval.

\begin{theorem}
\label{T:path main}
Consider the unit interval $\XF = [0,1]$ with its Euclidean metric $d$ and let $\mu$ be an absolutely continuous probability measure on $\XF$ with density $f$ satisfying
\begin{equation}
\label{density_interval}
c \leq f(x) \leq C\mbox{\rm \ for every $x \in \XF$},
\end{equation}
where $0 < c \leq C < \infty$. 
Given $\alpha \in (0, \infty)$ and $n\in\NB$, let
\begin{equation}
\label{rn_interval}
r(n)\equiv r(n,c, \alpha) := 4 c^{-1} \sqrt{(2 + \alpha) \frac{\ln n}{n}}.
\end{equation}
Then there exists $n_0\equiv n_0(c,C,\alpha)$ such that $n\geq n_0$ implies that $(\XF,d,\mu)$ is $(n^{-\alpha},r(n))$-disintegrable.
\end{theorem}

\begin{proof}[Proof of Theorem \ref{T:path main}]
Let $X_1,X_2,\dots$ be i.i.d.\ with distribution $\mu$ on a probability space $(\Psi,\FC,\PB)$, denote $\BC_n:= (B(X_1,r(n)), \ldots, B(X_n, r(n)))$ and recall the associated multifunction $F_{r(n),X_{[n]}}$ from \eqref{eq:Fn}.  We will show that
\begin{align}
H_n 
&= \{\mbox{$\BC_n$ does not cover~$\XF$}\} \label{Gnc_interval} \\
&{} \qquad {} \bigcup \{\mbox{$\BC_n$ covers~$\XF$ 
and there is no disintegration of $\mu$ along $F_{r(n),X_{[n]}}$}\} \nonumber
\end{align}
is contained in a set of $\PB$-probability at most $n^{ - \alpha}$.

One can handle the first of the two sets appearing in~\eqref{Gnc_interval} immediately.  Applying \cite[Corollary~2.9]{reznikov}, there exists $c_1 \equiv c_1(c, C, \alpha)$ and $n_1 \equiv n_1(c, C, \alpha)$ such that, for any choice of $r(n)$ satisfying
\begin{equation}
\label{first_interval}
r(n)\geq c_1 \frac{\ln n}{n}, \quad n \geq n_1,
\end{equation}
we have
\[
\PB(\mbox{$\BC_n$ does not cover~$\XF$}) \leq \frac{1}{2} n^{ - \alpha}, \quad n \geq n_1.
\]

Now we consider the second set in the union in~\eqref{Gnc_interval}.  Let $Y_1, \ldots, Y_n$ denote the order statistics for 
$X_1, \ldots, X_n$; let $\widetilde{B}_i := B(Y_i, r(n))$ for $i \in [n]$; and let $\widetilde{\BC}_n := (\widetilde{B}_1, \ldots, \widetilde{B}_n)$.  Of course, $\BC_n$ covers~$\XF$ if and only if $\widetilde{\BC}_n$ does; moreover, if these equivalent conditions are met, then, by the symmetry of the uniform distribution on $[n]$, there is an obvious one-to-one correspondence between disintegrations of $\mu$ with respect to $F_{r(n),X_{[n]}}$ and those with respect to $F_{r(n),Y_{[n]}}$.  Let $*_I$ denote the inequality appearing in~\eqref{nasc} when $\BC_n$ is replaced by $\widetilde{\BC}_n$.  For the second set in the union in~\eqref{Gnc_interval}, we observe
\begin{align}
\lefteqn{\hspace{-.2in}
\{\mbox{$\BC_n$ covers~$\XF$ and there is no disintegration of $\mu$ with respect to $F_{r(n),X_{[n]}}$}\}} \nonumber \\
&= 
\{\mbox{$\widetilde{\BC}_n$ covers~$\XF$ and there is no disintegration of $\mu$ with respect to $F_{r(n),Y_{[n]}}$}\} \nonumber \\
&= \cup_{I \in 2^{[n]}} \{\mbox{$\widetilde{\BC}_n$ covers~$\XF$ and $*_I$ fails}\} \nonumber \\
&= \cup_{\mbox{\scriptsize $\PC_n$-connected $I \in 2^{[n]}$}} \{\mbox{$\widetilde{\BC}_n$ covers~$\XF$ and $*_I$ fails}\} \nonumber \\
&= \bigcup_{\mbox{\scriptsize nonempty $\PC_n$-connected $I \subsetneqq [n]$}} 
\left\{ \mbox{$\widetilde{\BC}_n$ covers~$\XF$ and\ }\mu\left( \bigcup_{i \in I} \widetilde{B}_i \right) < \frac{|I|}{n} \right\}, \label{union_interval}
\end{align}
where the second equality holds by Theorem \ref{thm:existence}, the third equality holds by Corollary \ref{C:connected_interval}, and the fourth equality holds because the inequality in~\eqref{union_interval} cannot hold when~$I$ is empty nor, if $\widetilde{\BC}_n$ covers~$\XF$, when 
$|I| = n$.

Fixing~$n$, it will suffice to deal with each of the $\binom{n + 1}{2} - 1$ sets~$I$ appearing in~\eqref{union_interval} individually.  Fix such a set~$I$ and write
\[
I = \{ i_0 + 1, \ldots, i_0 + k\}
\] 
with $k \in [n - 1]$.
If $\widetilde{\BC}_n$ covers~$\XF$ and $\mu(\cup_{i \in I} \widetilde{B}_i) < k / n$ for such an~$I$, then 
\[
\mu(\cup_{i \in I} \widetilde{B}_i) = S_I + T_I,
\]
where 
\[
S_I = Y_{i_0 + k} - Y_{i_0 + 1}
\]
and $T_I$ is at least the $\mu$-value of a ball of radius $r(n)/ 2$, so that $T_I \geq \frac{1}{2} c r(n)$ by~\eqref{density_interval}.  Therefore, 
\begin{equation}
\label{bounding event_interval}
\left\{ \mbox{$\widetilde{\BC}_n$ covers~$\XF$ and\ }\mu\left( \bigcup_{i \in I} \widetilde{B}_i \right) < \frac{k}{n} \right\}
\subseteq \left\{ S_I < \frac{k}{n} - \frac{1}{2} c r(n)\right\} =: E \equiv E_{n, I}. 
\end{equation}

But (with $Y_0 := 0$ and $Y_{n + 1} := 1$) the joint distribution of the $\mu$-values of the $n + 1$ order-statistics spacings 
$Y_i - Y_{i - 1}$, $i \in [n + 1]$ is Dirichlet$(1, \ldots, 1)$; this is classical in the special case that~$\mu$ is the uniform distribution and follows in general from the fact that if $X \sim \mu$ satisfying~\eqref{density_interval} with distribution function~$F$, then $F(X)$ is uniformly distributed.  So the distribution of $S_I$ is unit mass at~$0$ if 
$k = 1$ and Beta$(k - 1, n + 1 - k)$ if $k \geq 2$.  If $k = 1$, then $\PB(E) = 0$ if $r(n)\geq 2 / (c n)$.  If $k \geq 2$, then $S_I$ has the distribution of the order statistic of rank $k - 1$ from a Uniform$(0, 1)$ sample of size~$n$, and so $\PB(E) = 0$ unless 
$k > \frac12 c n r(n)$, in which case 
\begin{equation}
\label{binomial_interval}
\PB(E) = \PB(V \geq k - 1),
\end{equation}
where $V \equiv V_{n, k} \sim \mbox{Binomial}\!\left( n, \frac{k}{n} - \frac12 c r(n)\right)$.

Assembling the information obtained thus far, we have seen that if $r(n)\geq c_2 \frac{\ln n}{n}$ for $n \geq n_1$, where $c_1$ is as at~\eqref{first_interval} and $c_2 := \max\{c_1, 2 / c\}$, then for $n \geq n_1$ the set $H_n$ is contained in an 
$\FC$-event of $\PB$-probability at most
\begin{align}
\lefteqn{\hspace{-0.5in}\frac12 n^{ - \alpha} + \sum_{k = \lfloor \frac12 c n r(n)\rfloor + 1}^{n - 1} (n + 1 - k)\,\PB(V_{n, k} \geq k - 1)} \nonumber \\
&\leq \frac12 n^{ - \alpha} + n \sum_{k = \lfloor \frac12 c n r(n)\rfloor + 1}^n \PB(V_{n, k} \geq k - 1). \label{two terms_interval}
\end{align}
To proceed, for the binomial probabilities here we use a Chernoff bound in the manageable form 
of~\cite[Corollary~4.4]{Mulzer}.  For $k \leq c n r(n)- 1$, the bound is 
\begin{equation}
\label{bound1_interval}
\exp[ - \tfrac14 \left( \tfrac12 c n r(n)- 1 \right)]; 
\end{equation}
otherwise, the bound is
\begin{equation}
\label{bound2_interval}
\exp\!\left[ - \frac14 \frac{\left( \frac12 c n r(n)- 1 \right)^2}{k - \frac12 c n r_n} \right] 
\leq \exp\!\left[ - \frac14 \frac{\left( \frac12 c n r(n)- 1 \right)^2}{n \left( 1 - \frac12 c r(n)\right)} \right]. 
\end{equation}
If $r(n)< 2 / c$, as we now assume, then one quickly verifies that~\eqref{bound2_interval} is the larger of the two bounds.  In that case, we can conclude for $n \geq n_1$ that $H_n$ is contained in a set of $\PB$-probability of at most
\begin{equation}
\label{two terms 2_interval}
\exp\!\left[ - \frac14 \frac{\left( \frac12 c n r(n)- 1 \right)^2}{n \left( 1 - \frac12 c r(n)\right)} \right].
\end{equation}
When $c n r(n)\geq 2$ and $c r(n)< 2$, the second term here is bounded by $\tfrac12 n^{ - \alpha}$ (where $\alpha > 0$) if (and only if) 
\begin{equation}
\label{messy_interval}
\frac12 c n r(n)\geq 2 [(2 + \alpha) \ln n + \ln 2] \left( \sqrt{\frac{n - 1}{(2 + \alpha) \ln n + \ln 2} + 1} - 1 \right) + 1.
\end{equation}
Since there exists $n_2 \equiv n_2(\alpha)$ such that the right side of~\eqref{messy_interval} is bounded by 
$2 \sqrt{(2 + \alpha) n \ln n}$ for all $n \geq n_2$, we now conclude that there exists $n_0 \equiv n_0(c, C, \alpha)$ such that
\[
r(n)= 4 c^{-1} \sqrt{(2 + \alpha) \frac{\ln n}{n}}
\]
is sufficient to bound~\eqref{two terms 2_interval} by $n^{ - \alpha}$ for all $n \geq n_0$.
\end{proof}

Essentially the same technique can be used to prove asymptotic disintegrability of the circle with the same rate, using the cycle graph $\ZB_n$ on $[n]$ in place of the path graph $\PC_n$.

\begin{theorem}\label{T:circle}
    Consider the unit circle $\XF\subset\RB^2$ equipped with geodesic metric $d$, normalised so that the diameter of $\XF$ is 1. Let $\mu$ be a probability measure on $\XF$ which is absolutely continuous with respect to the uniform probability measure, with density $f$ satisfying
    \[
    c\leq f(x)\leq C\text{ for every $x\in\XF$},
    \]
    where $0<c\leq C<\infty$. Given $\alpha\in(0,\infty)$ and $n\in\NB$, define
    \[
    r(n)\equiv r(n,c,\alpha):=2 c^{-1}\sqrt{(2+\alpha)\frac{\ln n}{n}}.
    \]
    Then there exists $n_0\equiv n_0(c,C,\alpha)$ such that $n\geq n_0$ implies $(\XF,d,\mu)$ is $(n^{-\alpha},r(n))$-disintegrable.
\end{theorem}

\subsection{Embeddings of finite connected simple graphs}\label{sec:graphs}

In this subsection, we show that embeddings of finite connected simple graphs in the plane are asymptotically disintegrable with similar rates to those of the interval and the circle. We will find the following general comparison result useful in dealing with embeddings of finite connected simple graphs.

\begin{proposition} \label{P:ordered metrics}
Consider two metrics $d_1$ and $d_2$ on a metric space $\XF$ which induce the same topology and hence the same Borel $\sigma$-algebra.  Let $\mu$ be a Borel probability measure on 
$\XF$.  Fix $n \geq 1$ and let~$\nu$ be a probability measure on $[n]$.  Given $x \in \XF$ and $r > 0$, write $B^{(k)}(x, r)$ for the closed $d_k$-ball of radius~$r$ centered at~$x$ for $k = 1,2$.   

Now fix $r > 0$, let $X_1, X_2, \ldots$ be i.i.d.\ with distribution~$\mu$ on a probability space $(\Psi, \FC, \PB)$, let 
$\BC^{(k)}_n := (B^{(k)}(X_1, r(n)), \ldots, B^{(k)}(X_n, r(n)))$ and let $F^{(k)}$ denote the corresponding multifunction defined as in \eqref{eq:Fn}, for $k = 1, 2$ . Assume that $d_1(x, y) \leq d_2(x, y)$ for $(\mu \times \mu)$-almost every $x, y \in \XF$.  Then 
\begin{align*}
\{& \text{$\BC^{(2)}$ covers $\XF$ and $(\mu_1, \ldots, \mu_n)$ is a disintegration of $\mu$ along $F^{(2)}$ with respect to $\nu$} \} \\
&\subseteq \{ \text{$\BC^{(1)}$ covers $\XF$ and $(\mu_1, \ldots, \mu_n)$ is a disintegration of $\mu$ along $F^{(1)}$ with respect to $\nu$} \}.
\end{align*}
\end{proposition}

\begin{proof}
This is quite easy.  Because $d_1 \leq d_2$, it follows that $B^{(2)}(x, r) \subseteq B^{(1)}(x, r)$ for every~$r$ and~$x$.  The result follows.
\end{proof}  

Now let $\GC = (V, E)$ be a given finite connected simple graph with $|V| \geq 2$ (equivalently, with $|E| \geq 1$).  Without any concern for edge-crossings, first embed~$\GC$ as a subset $\XF$ of $\RB^2$ arbitrarily using straight-line edges.

Define a metric $d(x, y)$ in a natural way by regarding each straight-line embedded edge as having length $1 / |E|$, so that $d(x, y)$ is then 
the length of the shortest path from~$x$ to~$y$ traversing only the embedded edges (with moves from one embedded edge to another permissible only at the embedded vertices $v \in V$).  This metric extends the notion of graph-distance, scaled by the factor $1 / |E|$.

In this general context, we now come to our disintegrability theorem for graphs.

\begin{theorem}
\label{T:graph main}
Suppose that~$\mu$ has density~$f$ with respect to uniform probability on~$\XF$ such that
\begin{equation}
\label{density_graph}
c \leq f(x) \leq C\mbox{\rm \ for every $x \in \XF$},
\end{equation}
where $0 < c \leq C < \infty$. 
Given $\alpha \in (0, \infty)$, let
\begin{equation}
\label{rn_graph}
r(n)\equiv r(n,c, \alpha) := 4 c^{-1} \sqrt{(2 + \alpha) \frac{\ln n}{n}}.
\end{equation}
Then there is $n_0\equiv n_0(c,C,\alpha)$ such that $n\geq n_0$ implies that $(\XF,d,\mu)$ is $([|E| (|E| - 1)]^{|E| - 1} n^{ - \alpha},r(n))$-disintegrable (with $0^0$ interpreted as $1$).
\end{theorem}

\begin{remark}
\label{R:two graph examples}
In particular, if $|V| = 2$ and $|E| = 1$, then $\XF$ can be taken to be the unit interval, as 
in Section \ref{sec:interval}, and then Theorem \ref{T:graph main} replicates Theorem \ref{T:path main}.
\end{remark}

\begin{proof}[Proof of Theorem \ref{T:graph main}]
A moment's reflection reveals that the theorem is equivalent to the same assertion when~$\XF$ is reduced by deleting the finite set consisting of all the embedded vertices.  So in proving the theorem we will sometimes, and without comment, consider~$\XF$ as reduced in that fashion.  This is useful because we will sometimes transform one graph to another, with the edges transformed in an isomorphic manner but with some vertices duplicated.  (See, for instance, Example \ref{E:triangle}.)

Theorem \ref{T:graph main} is established immediately from Corollary \ref{C:graphs to trees} (an immediate corollary of Lemma \ref{L:Omega tOmega}) and Proposition \ref{P:trees} to follow.
\end{proof}

To set up Lemma \ref{L:Omega tOmega}, consider the following construction and discussion.
Let $\GC = (V, E)$ be a given finite connected simple graph, and enumerate the vertices~$V$ as $v_1, \ldots, v_{|V|}$ in fixed but arbitrary fashion.  From an embedding~$\XF$ of~$\GC$ in $\RB^2$, construct an (embedded) tree~$\widetilde{T}$ as follows.  Choose arbitrarily and fix a spanning tree~$T$ for~$\GC$ and leave the embeddings of the edges in~$T$ undisturbed in creating~$\widetilde{T}$ and its embedding. For each pair $\{v_i, v_j\}$ of vertices in~$V$ with $i < j$ that are neighbours in $\GC$ but not neighbours in~$T$, delete the edge from $v_i$ to $v_j$ (and its embedding) and draw a line (segment) from $v_i$ to a(n arbitrarily chosen) point $v_{i, j}$ at the same distance from $v_i$ as $v_j$.  Delete all vertices from the embedding $\widetilde{\XF}$ of~$\widetilde{T}$.  

Then~$\widetilde{T}$ also has $|E|$ edges, and (before the deletion of vertices) its embedding~$\widetilde{\XF}$ has edge-lengths equal to those for~$\GC$.  In fact, there is an obvious one-to-one correspondence~$h$ between~$\XF$ and~$\widetilde{\XF}$.  Moreover, if~$d$ is the metric 
on~$\XF$ and~$\tilde{d}$ is the metric on~$\widetilde{\XF}$, then it is clear that
\begin{equation}
\label{ordered metrics}
d(x, y) \leq \tilde{d}(h(x), h(y))\mbox{\rm\ for $x, y \in \XF$}.
\end{equation}

\begin{example}
\label{E:triangle}
As a simple example, if~$\GC$ is a triangle embedded into~$\RB^2$ with vertices $v_1 = (0, 0)$, $v_2 = (1/3, 0)$, and 
$v_3 = (1/6, \sqrt{1/12})$ and the spanning tree~$T$ consists of the edges between $v_1$ and $v_2$ and between $v_1$ and $v_3$, then $\widetilde{T}$ has four vertices ($v_1$, $v_2$, $v_3$, and $v_{2, 3}$) and three edges ($\{v_1, v_2\}$, $\{v_1, v_3\}$, and 
$\{v_2, v_{2, 3}\}$), and the edge $\{v_2, v_{2, 3}\}$ in~$\widetilde{T}$ can be embedded as the line segment joining $(1/3, 0)$ and $(2/3, 0)$.  Illustrating~\eqref{ordered metrics}, if $x = \left( \frac{1}{12}, \frac{1}{2} \sqrt{\frac{1}{12}} \right) \in \XF$ and $y = \left( \frac{1}{4}, \frac{1}{2} \sqrt{\frac{1}{12}} \right) \in \XF$, then $h(x) = x \in \widetilde{\XF}$ and $h(y) = (\frac{1}{2}, 0) \in \widetilde{\XF}$, and
\[
d(x, y) = \tfrac{1}{3} < \tfrac{2}{3} = \tilde{d}(h(x), h(y)). 
\]
\end{example}

When combined with Proposition \ref{P:ordered metrics}, our discussion has established the following result.

\begin{lemma}\label{L:Omega tOmega}
If the conclusion of Theorem \ref{T:graph main} holds for~$\widetilde{\XF}$, then it holds for~$\XF$.~\qed 
\end{lemma}

We then obtain the following immediate corollary.

\begin{corollary}\label{C:graphs to trees}
If Theorem \ref{T:graph main} holds whenever $\GC = (V, E)$ is a finite tree, then it holds for any finite connected simple graph.
\end{corollary}

The proof of Theorem \ref{T:graph main} is thus reduced to the proof of the following proposition.

\begin{proposition}\label{P:trees}
Theorem \ref{T:graph main} holds whenever $\GC = (V, E)$ is a finite tree. 
\end{proposition}

\begin{proof}
In this proof, we may and do assume that the embedding of the given tree $T := \GC$ has actual length $1 / |E|$ for each embedded edge; subject to that constraint, the particular embedding chosen will not matter.  The proof will employ various transformations on tree-embeddings, and every such transformation will preserve $|V|$ (and therefore also $|E|$) and length $1 / |E|$ for each embedded edge.

First, Proposition \ref{P:trees} reduces to Theorem \ref{T:path main} when $|E| = 1$.  So we may assume $|E| \geq 2$.  

For any finite tree~$T$, let $\sigma(T)$ denote the number of nodes (i.e. vertices) of~$T$ \emph{not} having degree~$2$.  Note that $\sigma(T)$ is minimized over trees~$T$ with $|E|$ edges uniquely by the path-tree; 
then~$\XF$ is (without loss of generality) the interval $(0, |E|)$.  Further, we always have $\sigma(T) \leq |V| = |E| + 1$.  

Considering only embeddings of trees~$T$ with $|E|$ edges, let
\[
p_s := \max\!\left\{ \PB\!\left( \mu\!\left( \bigcup_{i \in I} B(X_i, r(n)) \right) < \frac{|I|}{n}\mbox{\ for some $I$} \right):
\sigma(T) \leq s \right\}.
\]
Then by Theorem \ref{T:path main} we have
\begin{equation}
\label{p2}
p_2 \leq n^{ - \alpha}.
\end{equation}
We will show that
\begin{equation}
\label{ps}
p_s \leq |E| (|E| - 1) p_{s - 1}, \quad 3 \leq s \leq |E| + 1.
\end{equation}
From \eqref{p2}--\eqref{ps} it follows immediately that
\begin{equation}
\label{psbound}
p_{|E| + 1} \leq [|E| (|E| - 1)]^{|E| - 1} n^{ - \alpha}\mbox{\ for all~$s$},
\end{equation}
and Proposition \ref{P:trees} is then established.

We now complete the proof of Proposition \ref{P:trees} by establishing~\eqref{ps}. Define
\[
\tilde{p}_s := \max\!\left\{ \PB\!\left( \mu\!\left( \bigcup_{i \in I} B(X_i, r(n)) \right) < \frac{|I|}{n}\mbox{\ for some $I$} \right):
\sigma(T) = s \right\}.
\]
Since $p_s = \max\{p_{s-1},\tilde{p}_s\}$, it is sufficient to establish
\begin{equation}
\label{pts}
\tilde{p}_s \leq |E| (|E| - 1) p_{s - 1}, \quad 3 \leq s \leq |E| + 1.
\end{equation}

Considering any tree~$T$ with $\sigma(T) = s \geq 3$, there must be a node of~$T$ having degree at least~$3$; choose one such node (arbitrarily) and root the tree at that node (call it~$\rho$).  That $\sigma(T)\geq 3$ moreover implies that there must be at least~$3$ leaves.  
Given such a tree~$T$ and any ordered pair $(\lambda_1, \lambda_2)$ of leaves in~$T$, consider the following transformation $\tau(\lambda_1,\lambda_2)$ (described using informal language, in an attempt to increase understanding) to another tree $T' \equiv T'(T, \lambda_1, \lambda_2)$.  Follow the path from~$\lambda_2$ to its first ancestor, call it~$\gamma_2$, that has degree at least~$3$;  note that $\gamma_2$ could be as close to $\lambda_2$ as $\lambda_2$'s parent or as distant as~$\rho$.  Detach the path joining $\lambda_2$ to $\gamma_2$ from $\gamma_2$ and reattach that path to the tree by gluing $\lambda_2$ to $\lambda_1$, thereby transforming the detached copy of $\gamma_2$ into a leaf in $T'$.  Note that $\sigma(T') \leq \sigma(T) - 1$.  [In fact, $\sigma(T') = \sigma(T) - 1$ unless $\gamma_2$ has degree~$3$, in which case $\sigma(T') = \sigma(T) - 2$.]

Now fix a subset $I\subseteq[n]$. Recalling again that since $\sigma(T)\geq 3$ there must be at least 3 leaves, either (i) there are at least two embedded leaves $\lambda$ with the property that no $X_i$ with $i \in I$ lies within distance $r(n)$ of~$\lambda$ or (ii) there are at least two embedded leaves~$\lambda$ with the property that there exists $X_i$ with $i \in I$ lying within distance $r(n)$ of~$\lambda$.  In either case, label the two leaves (in arbitrary order) as $(\lambda_1, \lambda_2)$.    Apply the transformation $\tau(\lambda_1,\lambda_2)$ described in the preceding paragraph to obtain 
$T'(\lambda_1,\lambda_2)$, equipped with the pushforward measure $\tau(\lambda_1,\lambda_2)_*\mu$.  Then, using the notation $B_T$ (respectively, $B_{T'}$) for balls in the tree~$T$ 
(respectively, $T'(\lambda_1,\lambda_2)$), one has  $\tau(\lambda_1,\lambda_2)^{-1}\cup_{i \in I} B_{T'}(X_i, r(n)) \subseteq\cup_{i \in I} B_T(X_i, r(n))$, and therefore 
$\tau(\lambda_1,\lambda_2)_*\mu(\cup_{i \in I} B_{T'}(X_i, r(n))) \leq \mu(\cup_{i \in I} B_T(X_i, r(n)))$.  Thus if 
\[
L(T) := \{(\lambda_1,\lambda_2):\mbox{\ $\lambda_1$ and $\lambda_2$ are leaves in~$T$}\},
\] 
then
\begin{align}
\lefteqn{\hspace{-0.5in}\PB\!\left( \mu\!\left( \bigcup_{i \in I} B_T(X_i, r(n)) \right) < \frac{|I|}{n}\mbox{\ for some $I$} \right)} \nonumber \\
&\leq \sum_{(\lambda_1,\lambda_2) \in L(T)} \PB\!\left( \tau(\lambda_1,\lambda_2)_*\mu\!\left( \bigcup_{i \in I} B_{T'(\lambda_1,\lambda_2)}(X_i, r(n)) \right)  < \frac{|I|}{n}\mbox{\ for some $I$} \right).
\label{Pbound}
\end{align}
Since there are at most $(|V| - 1) (|V| - 2) = |E| (|E| - 1)$ ordered pairs of leaves in~$T$, the right side of~\eqref{Pbound} is bounded by
\[
|E| (|E| - 1) p_{s - 1},
\]
and~\eqref{pts} follows.
\end{proof}

\subsection{Higher-dimensional examples}

The purpose of this section is to prove the high-probability asymptotic disintegrability of the unit $D$-cube with various metrics, for $D\geq 1$. Our proof technique consists in reduction to the one-dimensional case; we expect that as a consequence our rates are loose. We first consider the unit $D$-cube with the $\ell^{\infty}$ metric, generalising Theorem \ref{T:path main}.

\begin{theorem}
\label{T:dd main}
Let $\XF = [0, 1]^D$ be the unit cube in dimension~$D$, endowed with the $\ell^{\infty}$ metric $d$ and its Borel $\sigma$-field,
and let~$\mu$ be an absolutely continuous probability measure on~$\XF$ with density $f$ satisfying
\[
c\leq f(x)\leq C\mbox{\rm \ for every $x \in \XF$},
\]
where $0<c\leq C<\infty$. Given $\alpha \in (0, \infty)$, let
\begin{equation}
\label{rn_dd}
  r(n)\equiv r(n,D,c,\alpha) := c^{-1/D}(4 \cdot 3^{D - 1})^{1/D} (2 + \alpha)^{1/(2D)} \left( \frac{\ln n}{n} \right)^{1/(2 D)}.
\end{equation}
Then there exists $n_0\equiv n_0(D,c,C,\alpha)$ such that for every $n\geq n_0$, the space $(\XF,d,\mu)$ is $(n^{-\alpha},r(n))$-disintegrable.
\end{theorem}

\begin{proof}
As we shall see, a main ingredient of the proof is the \emph{one}-dimensional Theorem \ref{T:path main}, applied (with an idea similar to that in the proof of Proposition \ref{P:ordered metrics}) after a certain ``slicing and peeling'' of~$\XF$.  To set this up, let
\begin{equation}
\label{kndef}
k_n := \left\lfloor c^{1/D}\left( \frac34 \right)^{1/D} (2 + \alpha)^{-1/(2D)} \left( \frac{\ln n}{n} \right)^{-1/(2D)} \right\rfloor.
\end{equation}
The set $\XF$ may be partitioned into $k_n^{D-1}$ ``thin tubes", each of which is identifiable with a subset, excluding only some boundary faces, of $[0,1]\times [0,1/k_n]^{D-1}$. Thus $\XF$ may be identified, up to a subset of the measure-zero boundary, with $\widetilde{\XF}:=[0,k_n^{D-1}]\times[0,1/k_n]^{D-1}$ obtained by stacking the partitioning tubes of $\XF$ end to end, insisting only that the boundary face $\{1\}\times[0,1/k_n]^{D-1}$ (respectively, $\{0\}\times[0,1/k_n]^{D-1}$) of any tube is attached (if at all) only to the boundary face $\{1\}\times[0,1/k_n]^{D - 1}$ (resp., $\{0\}\times[0,1/k_n]^{D-1}$) via a reversal in orientation of one of the tubes. We deem two such partitioning tubes (thought of in $\XF$) to be \emph{adjacent} if and only if their ends are attached in $\widetilde{\XF}$; note then that, depending on the order in which the tubes are stacked, adjacent tubes in $\XF$ may or may not be close in the $\ell^{\infty}$ metric on the space of such tubes.

Denote by $S$ the map identifying (up to sets of measure zero) $\XF$ with $\widetilde{\XF}$. Note that $S$ carries the uniform probability measure on $\XF$ to the uniform probability measure on $\widetilde{\XF}$, hence sends $\mu$ to a probability measure $\tilde{\mu}$ on $\widetilde{\XF}$ which is absolutely continuous with density $\tilde{f}$ satisfying $c\leq \tilde{f}(x)\leq C$ for all $x\in S(\XF)$. Thus $(\XF,\mu)$ and $(\widetilde{\XF},\tilde{\mu})$ are identical, through $S$, up to a set of measure zero.

We claim moreover that, for any $x\in\XF$, the ball $B(x,r(n))$ is, for some $n_1\equiv n_1(D,c,\alpha)$ and all $n\geq n_1$, transformed by $S$ into a superset of $B_1(y_1,r(n))\times [0,1/k_n]^{D-1}$, where $y = (y_1,y_{2:D}) = S(x)$ and $B_1(y_1,r(n))$ is the ball of radius $r(n)$ centred at $y_1$ in the Euclidean distance on $[0,k_n^{D-1}]$. To establish this claim, it must be shown that if
\begin{equation}\label{w}
    w = (w_1,w_{2:D}) = S(z)\in\widetilde{\XF}\mbox{\, belongs to $B_1(y_1,r(n))\times[0,1/k_n]^{D-1}$, i.e., \ $|w_1-y_1|\leq r(n)$},
\end{equation}
then $z\in B(x,r(n))$. To demonstrate this, observe first that $x$ and $z$ belong either (i) to the same tube of $\XF$ or (ii) to adjacent tubes. In case (i), $|z_1-x_1| = |w_1-y_1|\leq r(n)$ and
\begin{equation}\label{coord2}
    \|z_{2:D}-x_{2:D}\|_{\ell^{\infty}} = \|w_{2:D}-y_{2:D}\|_{\ell^{\infty}}\leq \frac{1}{k_n} \leq \frac{1}{2}r(n)< r(n),
\end{equation}
where the penultimate inequality in \eqref{coord2} holds, for some $n_1\equiv n_1(D,c,\alpha)$ and all $n\geq n_1$, because, by \eqref{rn_dd} and \eqref{kndef},
\begin{equation}
k_n r(n) \rightarrow 3\mbox{ as $n\rightarrow\infty$}.
\end{equation}
Thus, in case (i), with~$d$ denoting $\ell^{\infty}$-distance in dimension~$D$, we have 
\begin{equation}\label{dxz}
d(x,z) = \max\{|z_1-x_1|,\|z_{2:D}-x_{2:D}\|_{\ell^{\infty}}\}\leq r(n),
\end{equation}
i.e. $z\in B(x,r(n))$, as claimed. In case (ii), supposing without loss of generality that $z_1$ and $x_1$ are both closer to $1$ than to $0$ in $[0,1]$ and that the tube containing $z$ is appended to the right (positive side) of the tube containing $x$, one has
\begin{equation}
    |z_1-x_1|\leq|z_1-1|+|1-x_1| = (1-z_1)+(1-x_1) = |w_1-y_1|\leq r(n),
\end{equation}
and $\|z_{2:D} - x_{2:D}\|_{\ell^{\infty}} \leq 2 / k_n \leq r(n)$ (c.f. \eqref{coord2}).  Thus, in case~(ii), as in case~(i), \eqref{dxz} holds, i.e., $z \in B(x, r(n))$, as claimed.

Having established the claim, let $X_1,X_2,\dots$ be i.i.d.\ with distribution $\mu$ on a probability space $(\Psi,\FC,\PB)$, let $\BC_n:=(B(X_1,r(n)),\dots,B(X_n,r(n)))$ and let $F_n:\XF \rightarrow [n]$ be the associated multifunction. We now see that the collection $\BC_n$ is transformed by~$S$ into an $n$-tuple $\widetilde{\BC}_n$ of supersets of random subsets $B_1(Y^{(i)}, r(n)) \times [0, 1 / k_n]^{D - 1}$, $i \in [n]$, of~$\widetilde{\XF}$, and we denote the associated multifunction by $\widetilde{F}_n:\widetilde{\XF}\rightarrow [n]$. Further, $Y^{(1)}, \ldots, Y^{(n)}$ here form an i.i.d.\ sample from the projection $\hat{\mu}:=(\pi_1)_*\tilde{\mu}$ of $\tilde{\mu}$ to $[0, k_n^{D-1}]$.  Note that $\hat{\mu}$ is absolutely continuous, with density uniformly lower-and-upper-bounded by $c/k_n^{D-1}$ and $C/k_n^{D-1}$ respectively. Let $\widehat{\BC}_n := (B_1(Y^{(1)}, r(n)), \ldots, B_1(Y^{(n)}, r(n)))$, and let $\widehat{F}_n:[0,k_n^{D-1}]\rightarrow [n]$ be the associated multifunction.  Then
\begin{align}
G_n:=\{&\mbox{\rm $\BC_n$ covers~$\XF$ and there exists a disintegration of $\mu$ along $F_n$}\}\nonumber \\
&\supseteq 
\{\mbox{\rm $\widetilde{\BC}_n$ covers~$\widetilde{\XF}$ and there exists a disintegration of $\tilde{\mu}$ along $\widetilde{F}_n$}\}\nonumber \\
&\supseteq 
\{\mbox{\rm $\widehat{\BC}_n$ covers~$[0, k_n^{D-1}]$ and there exists a disintegration of $\hat{\mu}$ along $\widehat{F}_n$}\}\nonumber \\
&=: \widehat{G}_n.\label{containment}
\end{align} 
Since, by~\eqref{rn_dd} and~\eqref{kndef},
\[
\frac{r_n}{k_n^{D-1}} \geq 4c^{-1} \sqrt{(2 + \alpha) \frac{\ln n}{n}},
\] 
it follows from a rescaling of \eqref{containment} by $1/k_n^{D-1}$ and Theorem \ref{T:path main} (taking into account also $n_1$), that there exists $n_0 \equiv n_0(D,c,C,\alpha)$ such that, for every $n \geq n_0$, the set $\widehat{G}_n$, and therefore also the set $G_n$, contains a set of $\PB$-probability at least 
$1 - n^{-\alpha}$.
\end{proof}

Because all norms on a finite-dimensional vector space are equivalent, by invoking Proposition \ref{P:ordered metrics}, Theorem \ref{T:dd main} can be extended to any metric induced from a norm (such as the $\ell^p$ metrics with $p \geq 1$) merely by altering the required radius~\eqref{rn_dd} by a constant factor.

Here is an example for $p = 2$ (the usual Euclidean metric) in general dimension~$D$.  Since the Euclidean norm is bounded above by $\sqrt{D}$ times the $\ell^{\infty}$ norm, we have the following immediate corollary to Theorem \ref{T:dd main}.

\begin{corollary}\label{C:non d main}
Let $\XF = [0, 1]^D$ be the unit cube in dimension~$D$, endowed with the Euclidean metric $d$ and its Borel $\sigma$-field.
Suppose that~$\mu$ is an absolutely continuous probability measure on~$\XF$ with density $f$ satisfying
\begin{equation}
\label{cor density_non d}
c \leq f(x) \leq C\mbox{\rm \ for every $x \in \XF$},
\end{equation} 
where $0 < c \leq C < \infty$. Given $\alpha \in (0, \infty)$, let
\begin{equation}
\label{cor rn_non d}
r(n)\equiv r(n,D, c, \alpha) := c^{-1/D} D^{1/2} (4 \cdot 3^{D - 1})^{1/D} (2 + \alpha)^{1/(2D)} \left( \frac{\ln n}{n} \right)^{1/(2 D)}.
\end{equation}
Then there exists $n_0 \equiv n_0(D, c, C, \alpha)$ such that, for every $n \geq n_0$, the space $(\XF,d,\mu)$ is $(n^{-\alpha},r(n))$-disintegrable.
\end{corollary}

\section{Wasserstein distance rates:\ examples}
\label{sec:wdre}

In this section we discuss how well the rates $r(n)$ obtained by combining the Euclidean-distance $\XF = [0, 1]^D$ asymptotic disintegrability results (one for each dimension~$D$) of Corollary~\ref{C:non d main} (which reduces to Theorem~\ref{T:path main} when $D = 1$) with Proposition~\ref{prop:wasserstein} fares in comparison to the best high-probability Wasserstein rates we can find in existing literature.  We shall assume throughout, as is customary, that the desired failure probability is $\epsilon(n) = n^{ - \alpha}$ with $\alpha > 0$; of course, both our rates and existing rates can be improved somewhat by allowing slower convergence to zero of $\epsilon(n)$.

There are many papers that discuss Wasserstein rates at various levels of generality; a nonexhaustive list is \cite{wasserstein1d, boissard, bolley_guillin_villani, fournier_guillin, lei, weedbach} and the references therein.  The results that we find give the best existing Wasserstein rates $r(n) \equiv r(n, D, c, C, \alpha)$ (up to multiplicative constants that may depend on $D$, $c$, $C$, and $\alpha$) for $(\XF, d, \mu)$ as in Corollary~\ref{C:non d main} for every dimension $D \in \NB$, every $p \in [1, \infty)$, and every $c$, $C$, and $\alpha$ are \cite[Theorems~3.1 and~5.4]{lei}, according to which the bound $W_p(\mu, \hat{\mu}_n) \leq r(n)$ holds with probability at least $1 - n^{- \alpha}$ for some sequence $(r(n))$ satisfying
\begin{equation} \label{best known w rates}
r(n) = O\!\left( n^{ - \frac{1}{\max(2p, D)}} (\ln n)^{{\bf 1}(p = D / 2) / p} + n^{- 1 / \max(2, p)} \sqrt{\ln n} \right).
\end{equation}
We then see that the power $1 / (2 D)$ of $n^{-1}$ in~\eqref{cor rn_non d} is never worse than the leading power of $n^{-1}$ in~\eqref{best known w rates} by more than a factor of~$2$.  Moreover, the two big-oh rates agree when $p = D = 1$, our rate is worse by only a logarithmic factor when $p = D \geq 2$, and our rate is \emph{better} by a factor of $p / D$ in the power of $n^{-1}$ when $p > D \geq 1$.

Of course to justify~\eqref{best known w rates} we need to verify that the conditions of \cite[Theorems~3.1 and~5.4]{lei} are met, and we do so next.  The applicability of the first of these two theorems (choosing $q > D$) is clear from the boundedness of the space $[0, 1]^D$.  For applicability of the second\footnote{In~\cite{lei}, it is claimed that the paper's Theorem~5.4 follows by ``[c]ombining the Lipschitz property of $W_p(\hat{\mu}, \mu)$ and the log Sobolev inequality under product measure [26], Theorem 5.3 and Corollary 5.7'', but there are no such numbered results either in the paper or in ``[26]'', which is our reference~\cite{ledoux}.  However, \cite[Theorem~5.4]{lei} follows by combining \cite[second paragraph of Section~5.3]{lei} with an unpublished argument of I.~Herbst explained in \cite[start of Section~2.3]{ledoux}.}, we need only note that the probability measures in Corollary~\ref{C:non d main} satisfy log-Sobolev inequalities.  This can be established by first noting that the $D$-dimensional standard normal density conditioned to $[0, 1]^D$ is a (so-called) strongly log-concave density on $\RB^D$ and therefore satisfies a log-Sobolev inequality and then (invoking~\eqref{cor density_non d}) using a perturbation technique that Ledoux attributes to Holley and Stroock~\cite{holley_stroock} and is described at the bottom of p.~200 in~\cite{ledoux} (but applied on the space $[0, 1]^D$ rather than $\RB^D$).

We believe it likely that our $D$-dimensional rate can be tightened to $W_p(\mu,\hat{\mu}_n) = O\!\left( ( \frac{\ln n}{n} )^{1/D} \right)$ for $D \geq 2$ and $p \in [1, \infty)$, but have discovered this to be a very difficult problem that we leave to future work.

\section{Conclusion and conjectures}

We have identified the notion and conditions of disintegration of a probability measure along a multifunction, and indicated its applications to providing (i)~high-probability bounds on Wasserstein distances to empirical measures and (ii)~high-probability average-case bounds on the error in approximating expectations of Lipschitz functions by empirical averages. We conclude by summarising the conjectures arising from our work.

Our rates for the asymptotic disintegrability of a metric Polish probability space $(\XF,d,\mu)$ yield (up to a multiplicative constant factor) the best known high-probability asymptotic rate for the Wasserstein distance $W_p(\mu,\hat{\mu}_n)$) when $\XF = [0,1]^D$ for all $p>D\geq  1$, but yield suboptimal rates when $D\geq 2$ and $1 \leq p\leq D$.
Our proof technique for higher-dimensional cubes, however, is based on the result for the $1$-dimensional case. We believe that a tighter analysis should be possible, enabling the (near-)optimal rate for $W_p(\mu,\hat{\mu}_n)$. More specifically, we conjecture the following.

\begin{conjecture}
    Let $\XF$ be any compact, connected, $D$-dimensional Riemannian manifold, $d$ its geodesic distance, and $\mu$ any probability measure that is absolutely continuous with respect to the volume measure, with density $f$ satisfying $0<c\leq f(x)\leq C<\infty$ for all $x\in\XF$. Then for any $\alpha\in(0,\infty)$, the space $(\XF,d,\mu)$ is asymptotically disintegrable with rate $\left( O(n^{-\alpha}),O\!\left( \left( \frac{\ln n}{n} \right)^{1 / \max\{2,D\}} \right) \right)$.
\end{conjecture}

We have also seen that asymptotic disintegrability (Definition \ref{def:disint2}) is sufficient for asymptotic average-case approximability (Definition \ref{def:averagecase}), and that the counterexample to the former supplied by Proposition \ref{prop:counterexample} is also a counterexample to the latter, for essentially the same reason (Proposition \ref{prop:averagecasecounterexample}). We thus conjecture the following.

\begin{conjecture}
    Let $(\XF,d,\mu)$ be a metric Polish probability space. Then $(\XF,d,\mu)$ is asymptotically average-case approximable if and only if $(\XF,d,\mu)$ is asymptotically disintegrable.
\end{conjecture}
\bigskip

\begin{acks}
We thank Amitabh Basu for helpful discussions. L.~E.~M.\ also thanks Ren\'{e} Vidal for helpful discussions.
\end{acks}

\bibliographystyle{plain}
\bibliography{references}

\end{document}